\documentclass{article}

\usepackage[all]{xy}
\usepackage{amsmath,amssymb,amsthm}

\usepackage{xcolor}

\newcommand{\Ring}{\mathbf{Ring}}

\newcommand{\Set}{\mathbf{Set}}

\newcommand{\mslattice}{\mathcal{S}^{\wedge}_{0,1}}
\newcommand{\dLat}{\mathbf{dLat}}

\newcommand{\V}{\textbf{V}}

\newtheorem{coro}{Corollary}
\newtheorem{defi}{Definition}

\newtheorem{rem}{Remark}
\newtheorem{prop}{Proposition}
\newtheorem{lem}{Lemma}

\newtheorem{theo}{Theorem}

\begin{document}

\title {A syntactical and categorical analysis for Varieties with Right Existentially Definable Factor Congruences}
\author{William Javier Zuluaga Botero}


\maketitle
\begin{abstract}
In this paper we prove that in the context of varieties with Right Existentially Definable Factor Congruences, definability of the property ``$\vec{e}$ and $\vec{f}$ are complementary central elements'', stability by complements and coextensivity of its respective algebraic category, are equivalent. 

\end{abstract}



\section{Introduction}\label{Introduction}

By a variety with $\vec{0}$ and $\vec{1}$ we understand a variety $\mathbf{V}$ for which there are $0$-ary terms $0_{1}$ , ..., $0_{n}$ , $1_{1}$ , ..., $1_{n}$ such that $\mathbf{V} \models \vec{0}\approx \vec{1}\rightarrow x\approx y$, where $\vec{0}=(0_{1}, ..., 0_{n})$ and $\vec{1}=(1_{1}, ..., 1_{n})$. If $\vec{a} \in A^{n}$ and $\vec{b} \in B^{n}$, we write $[\vec{a}, \vec{b}]$ for the n-uple $((a_{1} , b_{1} ), ..., (a_{n} , b_{n})) \in (A \times B)^{n}$. If $A\in \mathbf{V}$ then we say that $\vec{e}=(e_{1}, ..., e_{n})\in A^{n}$ is a \emph{central element} of $A$ if there exists an isomorphism	$\tau: A\rightarrow A_{1}\times A_{2}$, such that $\tau(\vec{e})=[\vec{0}, \vec{1}]$. Also, we say that $\vec{e}$ and $\vec{f}$ are a \emph{pair of complementary central elements} of $A$ if there exists an isomorphism $\tau: A\rightarrow A_{1}\times A_{2}$ such that $\tau(\vec{e})=[\vec{0}, \vec{1}]$ and $\tau(\vec{f})=[\vec{1}, \vec{0}]$. It is fairly known, that direct
product representations $ A\rightarrow A_{1}\times A_{2}$ of an algebra $A$ are closely related to the concept of factor congruence. A pair of congruences $(\theta, \delta)$ of an algebra $A$ is a pair of complementary factor congruences of $A$, if $\theta \cap \delta = \Delta$ and $\theta \circ \delta = \nabla$. In such a case, $\theta$ and $\delta$ are called \emph{factor congruences}. In most cases, the direct decompositions of an algebra are not unique; moreover, in general the pair $(\vec{e}, \vec{f})$ of complementary central elements does not determine the pair of complementary factor congruences  $(ker(\pi_{1}\tau ), ker(\pi_{2}\tau))$, where the $\pi_{i}'s$ are the canonical projections and $\tau$ is the isomorphism between $A$ and $A_{1}\times A_{2}$. We call such a property the \emph{determining property} (DP).
\begin{itemize}
\item[(DP)] For every pair $(\vec{e}, \vec{f})$ of complementary central elements, there is a unique pair $(\theta, \delta)$ of complementary factor congruences such that, for every $i = 1, ..., n$
\begin{center}
\begin{tabular}{ccc}
$(e_{i} , 0_{i} ) \in \theta$ and $(e_{i} , 1_{i} ) \in \delta$ & and & $(f_{i} , 0_{i} ) \in \delta$ and $(f_{i} , 1_{i} ) \in \theta$ 
\end{tabular}
\end{center}
\end{itemize}

Observe that (DP) is in some sense, the most general condition guaranteeing that central elements have all the information about direct product decompositions in the variety. 
\\

Let us consider the following properties.

\begin{enumerate}
\item[(L)] There exists a first order formula $\rho(\vec{z},x,y)$ such that, for every $A,B\in \mathbf{V}$, $a,b\in A$ and $c,d\in B$ 
\[A\times B\models \rho([\vec{0},\vec{1}], (a,c), (b,d))\;\textrm{iff}\; a=b \] 
\item[(R)] There exists a first order formula $\lambda(\vec{z},x,y)$ such that, for every $A,B\in \mathbf{V}$, $a,b\in A$ and $c,d\in B$ 
\[A\times B\models \lambda([\vec{0},\vec{1}], (a,c), (b,d))\;\textrm{iff}\; c=d \]
\end{enumerate} 

In \cite{SV2009} the following result was proved.
\begin{theo}\label{key Theorem about BFC}
Let $\mathbf{V}$ be a variety with $\vec{0}$ and $\vec{1}$. Then, each of the properties (DP), (L) and (R), are equivalent to $\mathbf{V}$ has Boolean factor congruences (BFC), i.e., the set of factor congruences of any algebra of $\mathbf{V}$ is a Boolean sublattice of its congruence lattice. 
\end{theo}

In the same paper it was proved that the existence of the formulas of items (L) and (R) are equivalent (which is not trivial, since in general $\vec{0}$ and $\vec{1}$ are not interchangeables). Nevertheless, such equivalence does not preserve the complexity of the formulas. To illustrate this situation, consider the variety $\mathcal{S}_{01}^{\vee}$ of bounded join semilattices. In \cite{BV2013}, it was proved that $\varphi(x,y,z)= x\vee z \approx y\vee z$ satisfies (L) but there is no positive nor existential formula satisfying (R). This fact, motivates the necessity of introduce several definitions in terms of the complexity of the formulas involved.
\\

We say that a variety $\mathbf{V}$ with BFC has \emph{right existentially definable factor congruences} (RexDFC) if the formula satisfying (R) is existential. Analogously, if the formula satisfying (L) is existential, we say that $\mathbf{V}$ has \emph{left existentially definable factor congruences} (LexDFC). If $\mathbf{V}$ has RexDFC and LexDFC, we say that $\mathbf{V}$ has \emph{twice existentially definable factor congruences} (TexDFC). Similar definitions arise when the considered formula is positive or equational (a finite conjunction of equations). When the formula is positive, we use the acronyms RpDFC, LpDFC and TpDFC to mean the variety has \emph{right positively definable factor congruences}, \emph{left positively definable factor congruences} and \emph{twice positively definable factor congruences}, respectively. 
\\

The following result, proved in \cite{S2010}, shows that there exist a relation linking some of the definability conditions above mentioned.

\begin{lem}[\cite{S2010}]\label{Existential implies Positive}
For every variety $\mathbf{V}$ with BFC the following holds:
\begin{enumerate}
\item RexDFC implies RpDFC.
\item LexDFC implies LpDFC.
\item TexDFC implies TpDFC.
\end{enumerate} 
\end{lem}

The aim of this work is to show the link that there exists between determined syntactical, structural and categorical conditions in a variety with RexDFC. Namely, we prove that definability of the property ``$\vec{e}$ and $\vec{f}$ are complementary central elements'', stability by complements and the coextensivity of its respective algebraic category, are equivalent.
\\

This work is organized as follows. The contents of Section \ref{Preliminaries}, are devoted to introduce the basics in varieties with BFC which are necessary to read this work. In Section \ref{A characterization for RexDFC}, we characterize varieties with RexDFC as those varieties with BFC such that the factor congruences are basically, principal congruences. In Section \ref{Coextensivity,  definability and stabitlity by complements}, we prove the main result of this work (Theorem \ref{scc is equivalent to coextensivity}) in which we show that for varieties with RexDFC, stability by complements, definability of $\vec{e}\diamond_{A}\vec{f}$ by a $\exists \bigwedge p=q$-formula and coextensivity of the algebraic category associated to the variety, are equivalent. Finally, in Section \ref{RexDFC and stability by complements induce homomorphisms of Boolean algebras} we prove that varieties with RexDFC stable by complements have the Fraser-Horn Property and, in particular, that this fact entails that homomorphisms, when restricted to central elements, are
Boolean algebra homomorphisms.

For standard notions in universal algebra the reader may consult \cite{MMT1987}.

\section{Preliminaries}\label{Preliminaries}

\subsection{Notation and basic results}\label{Notation and basic results}

If $A$ is an algebra of a given type, we denote the congruence lattice of $A$ by $Con(A)$. As usual, the join operation of $Con(A)$ is denoted by $\vee$. If $f:A\rightarrow B$ is a homomorphism and $n$ is a fixed natural number, we write $f$ to denote the application $f^{n}:A^{n}\rightarrow B^{n}$. Thereby, if $\vec{e}\in A^{n}$, then $f(\vec{e})$ must be taken as $(f(e_{1}),...,f(e_{n}))\in B^{n}$. We write $Ker(f)$ for the congruence of $A$, defined by $\{(a,b)\in A\times A\mid f(a)=f(b)\}$. The universal congruence on $A$ is denoted by $\nabla^{A}$ and $\Delta^{A}$ denotes the identity congruence on A (or simply $\nabla$ and $\Delta$ when the context is clear). If $S\subseteq A$, we write $\theta^{A}(S)$ for the least congruence containing $S\times S$. If $\vec{a}, \vec{b} \in A^{n}$,
then $\theta^{A}(\vec{a}, \vec{b})$ denotes the congruence generated by $C=\{(a_{i}, b_{i}) \mid 1 \leq i \leq n\}$. If $\vec{a}, \vec{b} \in A^{n}$ and $\theta \in Con(A)$, we write $[\vec{a},\vec{b}]\in \theta$ to express that $(a_{i} , b_{i}) \in \theta$, for
$i = 1,..., n$. We use $FC(A)$ to denote the set of factor congruences of A. A
variety $\mathbf{V}$ with $\vec{0}$ and $\vec{1}$ has BFC, if satisfies any of the equivalent conditions of Theorem \ref{key Theorem about BFC}. We write $\theta \diamond \delta$ in $Con(A)$ to denote that $\theta$ and $\delta$ are complementary factor congruences of $A$. If $\theta \in FC(A)$, we use $\theta^{\star}$ to denote the factor complement of $\theta$. If $\theta, \delta\in Con(A)$ we say that $\theta$ and $\delta$ \emph{permute} if $\theta \circ \delta= \delta \circ \theta$.
\\

A \emph{system over} $Con(A)$ is a $2k$-ple $(\theta_{1},...,\theta_{k}, x_{1},...,x_{k})$ such that $(x_{i},x_{j})\in \theta_{i}\vee\theta_{j}$, for every $i,j$. A \emph{solution} of the system $(\theta_{1},...,\theta_{k}, x_{1},...,x_{k})$ is an element $x\in A$ such that $(x,x_{i})\in \theta_{i}$ for every $1\leq i\leq k$. Observe that if $\theta_{1}\cap ... \cap \theta_{k}=\Delta^{A}$, then the system $(\theta_{1},...,\theta_{k}, x_{1},...,x_{k})$ has at most one solution.

\begin{lem}\label{basics about systems} 
Let $\theta$ and $\delta$ be congruences of $A$. The following are equivalent:
\begin{enumerate}
\item $\theta$ and $\delta$ permute.
\item $\theta \vee \delta =\theta \circ \delta$
\item For every $x,y\in A$, the system $(\theta, \delta, x,y)$ has a solution.
\end{enumerate}
\end{lem}

Given two sets $A_{1}$,$A_{2}$ and a relation $\delta$ in $A_{1} \times A_{2}$,  we say that $\delta$ \emph{factorizes} if there exist sets $\delta_{1} \subseteq A_{1} \times A_{1}$ and $\delta_{2} \subseteq A_{2} \times A_{2}$ such that $\delta=\delta_{1}\times \delta_{2}$, where \[\delta_{1}\times \delta_{2}=\{((a,b),(c,d))\mid (a,c)\in \delta_{1}, (b,d)\in \delta_{2}\}\]
So, if $\delta\in Con(A_{1}\times A_{2})$ factorizes in $\delta_{1}, \delta_{2}$ it follows that $\delta_{i}\in Con (A_{i})$, for $i=1,2$.
\begin{lem}[\cite{BB1990}]\label{FC Factors equivalent BFC}
Let $\V$ be a variety. The following are equivalent:
\begin{enumerate}
\item $\V$ has BFC.
\item $\V$ has factorable factor congruences. I.e. If $A,B\in \V$ and $\theta\in FC(A\times B)$, then $\theta$ factorizes.
\end{enumerate}
\end{lem}

We say that a variety has the \emph{Fraser-Horn property} (see \cite{FH1970}) (FHP) if every congruence on a (finite) direct product of algebras factorizes. Given a variety $\V$ and a set of variables $X$, we use $\textbf{F}_{\V}(X)$ to denote the free algebra of $\V$ freely generated by $X$ (or simply $\textbf{F}(X)$, if the context is clear). If $X = \{x_{1} , . . . , x_{n} \}$, then we use $\textbf{F}_{\V}(x_{1} , . . . , x_{k})$ instead of $\textbf{F}_{\V} (\{x_{1} , . . . , x_{k} \})$. 
\\

As a final remark, we should recall that all the algebras considered along this work always will be assumed as algebras with finite $m$-ary function symbols and its type (unless necessary), will be omitted.

\subsection{Generalities about Varieties with BFC}\label{Generalities about Varieties with DFC}


Let $\mathbf{V}$ be a variety with $\vec{0}$ and $\vec{1}$ and suppose that has BFC. For every $A\in \mathbf{V}$, we write $Z(A)$ to denote the set of central elements of $A$ and $\vec{e}\diamond_{A} \vec{f}$ to denote that $\vec{e}$ and $\vec{f}$ are complementary central elements of $A$. If $\vec{e}$ is a central element of $A$ we write $\theta_{\vec{0}, \vec{e}}^{A}$ and $\theta_{\vec{1}, \vec{e}}^{A}$ for the unique pair of complementary factor congruences satisfying $[\vec{e}, \vec{0}] \in \theta_{\vec{0}, \vec{e}}^{A}$ and  $[\vec{e}, \vec{1}]\in \theta_{\vec{1}, \vec{e}}^{A}$. It follows that $\vec{0}$ and $\vec{1}$ are central elements in every algebra $A$ and the factor congruences associated to them are $\theta_{\vec{0}, \vec{0}}^{A}=\Delta^{A}$, $\theta_{\vec{1}, \vec{0}}^{A}=\nabla^{A}$ and $\theta_{\vec{0}, \vec{1}}^{A}=\nabla^{A}$, $\theta_{\vec{1}, \vec{1}}^{A}=\Delta^{A}$, respectively. If there is no place to confusion, we write $\theta_{\vec{0}, \vec{e}}^{A}$ and $\theta_{\vec{1}, \vec{e}}^{A}$ simply as $\theta_{\vec{0}, \vec{e}}$ and $\theta_{\vec{1}, \vec{e}}$. Since $\mathbf{V}$ has BFC, factor complements are unique so we obtain the following fundamental result

\begin{theo}\label{Bijection betwen centrals and factor congruences}
Let $\mathbf{V}$ be a variety with DFC. The map $g:Z(A)\rightarrow FC(A),$ defined by $g(\vec{e})=\theta_{\vec{0}, \vec{e}}^{A}$ is a bijection and its inverse $h:FC(A)\rightarrow Z(A)$ is defined by $h(\theta)=\vec{e}$, where $\vec{e}$ is the only $\vec{e}\in A^{n}$ such that $[\vec{e}, \vec{0}]\in \theta$ and $[\vec{e}, \vec{1}]\in \theta^{\ast}$.
\end{theo}

As a consequence of Lemma \ref{basics about systems}, we obtain the following result for varieties with BFC: 

\begin{lem}\label{BFC has permutable congruences} 
In every algebra $A$ of a variety $\mathbf{V}$ with BFC, every pair of factor congruences permute. 
\end{lem}

These facts, allows us to define some operations in $Z(A)$ as follows: Given $\vec{e}\in Z(A)$, the \emph{complement $\vec{e}^{c_{A}}$} of $\vec{e}$, is the only solution to the equations $[\vec{z}, \vec{1}]\in \theta_{\vec{0},\vec{e}}$ and $[\vec{z}, \vec{0}]\in \theta_{\vec{1},\vec{e}}$. Given $\vec{e}, \vec{f}\in Z(A)$, the \emph{infimum} $\vec{e}\wedge_{A}\vec{f}$ is the only solution to the equations $[\vec{z}, \vec{0}]\in \theta_{\vec{0},\vec{e}}\cap \theta_{\vec{0},\vec{f}}$ and $[\vec{z}, \vec{1}]\in \theta_{\vec{1},\vec{e}}\vee \theta_{\vec{1},\vec{f}}$. Finally, the \emph{supremum} $\vec{e}\vee_{A}\vec{f}$ is the only solution to the equations $[\vec{z}, \vec{0}]\in \theta_{\vec{0},\vec{e}}\vee \theta_{\vec{0},\vec{f}}$ and $[\vec{z}, \vec{1}]\in \theta_{\vec{1},\vec{e}}\cap \theta_{\vec{1},\vec{f}}$. 
\\

As result, we obtain that $\textbf{Z}(A)=(Z(A),\wedge_{A},\vee_{A}, ^{c_{A}},\vec{0},\vec{1})$ is a Boolean algebra which is isomorphic to $(FC(A), \vee, \cap, ^{\ast},\Delta^{A},\nabla^{A})$. Also notice that $\vec{e}\leq_{A} \vec{f}$ if and only if $\theta^{A}_{\vec{0},\vec{e}}\subseteq \theta^{A}_{\vec{0},\vec{f}}$, which in turn, is equivalent to $\theta^{A}_{\vec{1},\vec{f}}\subseteq \theta^{A}_{\vec{1},\vec{e}}$. If the context is clear, we will omit the subscripts in the operations of $\textbf{Z}(A)$.

\begin{lem}\label{Useful lema Centrals}
Let $\V$ be a variety with BFC and $A\in \V$. For every $\vec{e},\vec{f}\in Z(A)$, the following holds:
\begin{enumerate}
\item $\vec{a}=\vec{e}\wedge_{A}\vec{f}$ if and only if $[\vec{0},\vec{a}]\in \theta_{\vec{0},\vec{e}}$ and $[\vec{a},\vec{f}]\in \theta_{\vec{1},\vec{e}}$.
\item $\vec{a}=\vec{e}\vee_{A}\vec{f}$ if and only if $[\vec{1},\vec{a}]\in \theta_{\vec{1},\vec{e}}$ and $[\vec{a},\vec{f}]\in \theta_{\vec{0},\vec{e}}$.
\end{enumerate}
\end{lem}

\section{A characterization for RexDFC}\label{A characterization for RexDFC}

In this section we prove that a variety with BFC has RexDFC if and only if the factor congruence $\theta_{\vec{1},\vec{e}}$ associated to a central element $\vec{e}$, coincides with the principal congruence that identifies $\vec{1}$ with $\vec{e}$. 
\\

We begin by recalling the following (Gr\"atzer) version of Maltsev's key observation on principal congruences.

\begin{lem}\label{Gratzer Malsev Lemma} Let $A$ be an algebra and $a, b \in \textbf{A}$, $\vec{c}$, $\vec{d}$ $\in A^{n}$. Then $(a,b)\in \theta^{A}(\vec{c},\vec{d})$ if and only if there exist $(n+m)$-ary terms $t_{1}(\vec{x},\vec{u})$,...,$t_{k}(\vec{x},\vec{u})$ with $k$ odd and $\vec{\lambda}\in A^{m}$ such that: 

\begin{center}
\begin{tabular}{cc}
$a=t_{1}(\vec{c},\vec{\lambda})$ & $b=t_{k}(\vec{d},\vec{\lambda})$
\\
$t_{i}(\vec{c},\vec{\lambda})=t_{i+1}(\vec{c},\vec{\lambda})$, $i$ even, & $t_{i}(\vec{d},\vec{\lambda})=t_{i+1}(\vec{d},\vec{\lambda}),$ $i$ odd.
\end{tabular}
\end{center}
\end{lem}

A \emph{principal congruence formula} is a formula $\pi(x,y,\vec{u},\vec{v})$ of the form
\[\exists_{\vec{w}}(x\approx t_{1}(\vec{u},\vec{w})\wedge \bigwedge_{i\in E_{k}} (t_{i}(\vec{u},\vec{w})\approx t_{i+1}(\vec{u},\vec{w}))\wedge \bigwedge_{i\in O_{k}}(t_{i}(\vec{v},\vec{w})\approx t_{i+1}(\vec{v},\vec{w}))\wedge t_{k}(\vec{v},\vec{w})\approx y)), \]

where $k$ is odd and $t_{i}$ are terms of type $\tau$. This fact allows us to restate Lemma \ref{Gratzer Malsev Lemma} as follows:

\begin{lem}\label{Malsev restated}
Let $A$ be an algebra, $a, b \in \textbf{A}$, $\vec{c}$, $\vec{d}$ $\in A^{n}$. Then $(a,b)\in \theta^{A}(\vec{c},\vec{d})$ if and only if there exists a principal congruence formula $\pi$, such that $A\models \pi(a,b,\vec{c},\vec{d})$.
\end{lem}

The following result (Claim 2 after Proposition 6.6 in \cite{SV2009}) will be useful for proving the main result of this section. 

\begin{lem}\label{Positive implies principal congruence}
Let $\mathbf{V}$ be a variety with BFC. If $\mathbf{V}$ has RpDFC, then for every $A\in \mathbf{V}$ and $\vec{e}\in Z(A)$, $\theta_{\vec{1}, \vec{e}}^{A}=\theta^{A}(\vec{1},\vec{e})$. 
\end{lem}

\begin{lem}\label{Definability by principal congruences}
Let $\mathbf{V}$ be a variety with BFC. Then, $\V$ has RexDFC if and only if for every $A\in \mathbf{V}$ and $\vec{e}\in Z(A)$, $\theta^{A}_{\vec{1},\vec{e}}=\theta^{A}(\vec{1},\vec{e})$. 
\end{lem}

\begin{proof}

The first implication follows directly from Lemmas \ref{Existential implies Positive} and \ref{Positive implies principal congruence}. On the other hand, in order to prove the second implication, let us write $P=\textbf{F}(x,y)\times \textbf{F}(y)$, where $\textbf{F}(x,y)$ and $\textbf{F}(y)$ are the free algebras generated by $\{x,y\}$ and $\{y\}$, respectively. By hypothesis, $Ker(\pi_{2})=\theta^{\bf{P}}_{[\vec{1},\vec{1}],[\vec{0},\vec{1}]}=\theta^{P}([\vec{1},\vec{1}],[\vec{0},\vec{1}])$. Since the pair $((x,y),(y,y))\in Ker(\pi_{2})$, from Lemma \ref{Gratzer Malsev Lemma}, there exist $(n+m)$-ary terms $t_{1}(\vec{x},\vec{u})$,...,$t_{k}(\vec{x},\vec{u})$ with $k$ odd and $\vec{u}\in P^{m}$ such that:

\begin{equation}
\begin{array}{cc}
(x,y)=t^{P}_{1}[[\vec{1},\vec{1}],\vec{u}]& (y,y)=t^{P}_{k}[[\vec{0},\vec{1}],\vec{u}]
\\
t^{P}_{i}[[\vec{1},\vec{1}],\vec{u}]=t^{P}_{i+1}[[\vec{1},\vec{1}],\vec{u}],\; i\in E_{k}, & t^{P}_{i}[[\vec{0},\vec{1}],\vec{u}]=t^{P}_{i+1}[[\vec{0},\vec{1}],\vec{u}],\; i\in O_{k}.\label{1}
\end{array}
\end{equation}

where $E_{k}$ and $O_{k}$ refer to the even and odd naturals less or equal to $k$, respectively.
\\

Since $\vec{u}\in P$, there are $\vec{P}(x,y)\in F(x,y)$ and $\vec{Q}(y)\in F(x,y)$, such that $\vec{u}=[\vec{P}, \vec{Q}]$. Recall that $t^{P}_{i}[[\vec{R},\vec{S}],[\vec{P}, \vec{Q}]]=(t^{F(x,y)}_{i}[\vec{R},\vec{P}],t^{F(y)}_{i}[\vec{S},\vec{Q}])$, for $1\leq i \leq k$ and $[\vec{R},\vec{S}]\in P$, thus, from equation (\ref{1}), we obtain that there exist $(n+m)$-ary terms $t_{1}(\vec{x},\vec{u})$,...,$t_{k}(\vec{x},\vec{u})$ with $k$ odd, $\vec{P}(x,y)\in F(x,y)$ and $\vec{Q}(y)\in F(y)$, such that:  

{\small
\begin{displaymath}
y=t^{F(y)}_{i}[\vec{0},\vec{Q}(y)], \textrm{for every}\; 1\leq i\leq k 
\end{displaymath}
}
and

{\small
\begin{displaymath}
\begin{array}{cc}
x=t^{F(x,y)}_{1}[\vec{1},\vec{P}(x,y)] & y=t^{F(x,y)}_{k}[\vec{0},\vec{P}(x,y)]
\\
t^{F(x,y)}_{i}[\vec{1},\vec{P}(x,y)]=t^{F(x,y)}_{i+1}[\vec{1},\vec{P}(x,y)],\; i\in E_{k} & t^{F(x,y)}_{i}[\vec{0},\vec{P}(x,y)]=t^{F(x,y)}_{i+1}[\vec{0},\vec{P}(x,y)],\; i\in O_{k}
\end{array}
\end{displaymath}
}  
  
Let $\varphi(x,y,\vec{z})=\pi(x,y,\vec{1},\vec{z})$. In order to to check that $\varphi$ defines $\theta^{A}_{\vec{1},\vec{e}}$ in terms of $\vec{e}$ let us assume $A,B\in \mathcal{V}$ and $(a,b),(c,d)\in A\times B$. Since the free algebra functor $\textbf{F}:\Set \rightarrow \mathcal{V}$ is left adjoint to the forgetful functor, in the case of $b=d$, the assignments $\alpha_{A}:\{x,y\}\rightarrow A$ and $\alpha_{B}:\{y\}\rightarrow B$, defined by $\alpha_{A}(x)=a$, $\alpha_{A}(y)=c$ and $\alpha_{B}(y)=b$ generate a unique pair of homomorphisms $\beta_{A}:\textbf{F}(x,y)\rightarrow A$ and $\beta_{B}:\textbf{F}(y)\rightarrow B$ extending $\alpha_{A}$ and $\alpha_{B}$, respectively. Therefore, since $P\models \varphi((x,y),(y,y), [\vec{0},\vec{1}])$, by applying $g=\beta_{A}\times \beta_{B}$ in (\ref{1}), we obtain as result that $A\times B\models \varphi((a,b),(c,b),[\vec{0},\vec{1}])$. On the other hand, if $A\times B\models \varphi((a,b),(c,d),[\vec{0},\vec{1}])$, then there exist $[\vec{\varepsilon},\vec{\delta}]\in A\times B$, such that 

{\small
\begin{equation}\label{2}
\begin{array}{cc}
(a,b)=t^{A\times B}_{1}[[\vec{1},\vec{1}],[\vec{\varepsilon},\vec{\delta}]] & (c,d)=t^{A\times B}_{k}[[\vec{0},\vec{1}],[\vec{\varepsilon},\vec{\delta}]]
\\
t^{A\times B}_{i}[[\vec{1},\vec{1}],[\vec{\varepsilon},\vec{\delta}]]=t^{A\times B}_{i+1}[[\vec{1},\vec{1}],[\vec{\varepsilon},\vec{\delta}]],\;i\in E_{k}, & t^{A\times B}_{i}[[\vec{0},\vec{1}],[\vec{\varepsilon},\vec{\delta}]]=t^{A\times B}_{i+1}[[\vec{0},\vec{1}],[\vec{\varepsilon},\vec{\delta}]],\;i\in O_{k}.
\end{array}
\end{equation}
}

So, since $t^{A\times B}_{i}[[\vec{j},\vec{r}],[\vec{\varepsilon},\vec{\delta}]]=(t^{A}_{i}[\vec{j},\vec{\varepsilon}],t^{B}_{i}[\vec{r},\vec{\delta}])$, for every $[\vec{j},\vec{r}]\in A\times B$ and $1\leq i\leq k$, from (\ref{2}), we conclude that

\begin{displaymath}
\begin{array}{cc}
b=t^{B}_{1}[\vec{1},\vec{\varepsilon}] & d=t^{B}_{k}[\vec{1},\vec{\varepsilon}]
\\
t^{B}_{i}[\vec{1},\vec{\varepsilon}]=t^{B}_{i+1}[\vec{1},\vec{\varepsilon}],\; i\in E_{k} & t^{B}_{i}[\vec{1},\vec{\varepsilon}]=t^{B}_{i+1}[\vec{1},\vec{\varepsilon}],\; i\in O_{k}.
\end{array}
\end{displaymath}

Which by Lemma \ref{Gratzer Malsev Lemma} means that $(b,d)\in \theta^{B}(\vec{1},\vec{1})$. Since $\theta^{B}(\vec{1},\vec{1})=\theta^{B}_{\vec{1},\vec{1}}$ by assumption and $\theta^{B}(\vec{1},\vec{1})=\Delta^{B}$, we get that $b=d$. This concludes the proof.

\end{proof}

\begin{coro}\label{Centrals are determined by principal congruences}
Let $\textbf{V}$ be a variety with RexDFC. Then, for every $A\in \textbf{V}$, the map $Z(A)\rightarrow FC(A)$, defined by $\vec{e}\mapsto \theta^{A}(\vec{1},\vec{e})$ is bijective.
\end{coro}

\section{Coextensivity, definability and stability by complements}\label{Coextensivity,  definability and stabitlity by complements}

\subsection{The universal property}\label{The universal property}

In this section we present some useful results that raise from the universal property of principal congruences in varieties with BFC.
\\

Let $A$ and $P$ algebras of the same type. We say that a homomorphism $f:A \rightarrow P$ has the \emph{universal property of identify the elements of $S$}, if for every homomorphism $g:A\rightarrow C$, such that $g(a)=g(b)$, for every $a,b\in S$; there exists a unique homomorphism $h:B\rightarrow C$, making the diagram

\begin{displaymath}
\xymatrix{
A \ar[r]^-{f} \ar[dr]_-{g} & B \ar@{-->}[d]^-{h}
\\
 & C
}
\end{displaymath}
\noindent
commutes.
\\

The following results are key observations for the rest of this paper. Since their proofs are routine, we leave to the reader the details of its proof.

\begin{lem}\label{universal property principal congruences} Let $A$ be an algebra and $S\subseteq A$. Then, the canonical homomorphism $\nu_{S}:A\rightarrow A/\theta(S)$ has the universal property of identify all the elements of $S$. 
\end{lem}

\begin{lem}\label{diagram pushout}
Let $A$ and $B$ be algebras of the same type and $f:A\rightarrow B$ be a homomorphism. Then, for every $S\subseteq A$, the diagram
\begin{displaymath}
\xymatrix{
A \ar[r]^-{\nu_{S}} \ar[d]_-{f} & A/\theta^{A}(S) \ar[d]
\\
B \ar[r]_-{\nu_{f(S)}} & B/\theta^{B}(f(S))
}
\end{displaymath}
\noindent
is a pushout. 
\end{lem}

Recall that, as a consequence of Lemma \ref{universal property principal congruences}, it follows that, for every $\vec{a},\vec{b}\in A^{n}$, the canonical homomorphism $A\rightarrow A/\theta(\vec{a},\vec{b})$ has the universal property of identify the pairs $(a_{k},b_{k})$, for $1\leq k\leq n$. Hence, as a straight consequence of Lemmas \ref{Definability by principal congruences}, \ref{universal property principal congruences} and \ref{diagram pushout} we get the following result. 

\begin{coro}\label{corollary universal property} 
Let $\mathbf{V}$ be a variety with RexDFC. If $A,B\in \mathbf{V}$ and $\vec{e}\in Z(A)$, then:
\begin{enumerate}
\item  The canonical homomorphism $A\rightarrow A/\theta^{A}_{\vec{1},\vec{e}}$ has the universal property of identify $\vec{e}$ with $\vec{1}$.
\item For every homomorphism $f:A\rightarrow B$, the diagram

\begin{displaymath}
\xymatrix{
A \ar[r]^-{\nu_{e}} \ar[d]_-{f} & A/\theta^{A}(\vec{1},\vec{e}) \ar[d]
\\
B \ar[r]_-{\nu_{f(e)}} & B/\theta^{B}(\vec{1},f(\vec{e}))
}
\end{displaymath}

is a pushout. 
\end{enumerate}
\end{coro}

\subsection{The equivalence}

In the context of varieties with BFC one may be tempted to think that in general, homomorphisms preserve central elements and even complementary central elements. Unfortunately, as we will see, this is not case. In this section we present a characterization for varieties with RexDFC such that their homomorphisms preserve complementary central elements in terms of a certain definability condition of the relation $\vec{e}\diamond_{A}\vec{f}$ and a purely categorical property, namely the coextensivity.
\\

Let $\mathbf{V}$ be a variety with BFC. If $A,B\in \mathbf{V}$ and $f:A\rightarrow B$ is a homomorphism, we say that $f$ \emph{preserves central elements} if the map $f:Z(A)\rightarrow Z(B)$ is well defined; that is to say, for every $\vec{e}\in Z(A),$ it follows that $f(\vec{e})\in Z(B)$. We say that $f$ \emph{preserves complementary central elements} if preserves central elements and for every $\vec{e}_{1}, \vec{e}_{2}\in Z(A)$, 
\[\vec{e}_{1}\diamond_{A}\vec{e}_{2} \Rightarrow f(\vec{e}_{1})\diamond_{B}f(\vec{e}_{2}). \]

\begin{rem}\label{SC and CSC are not trivial}
Observe that, since complements of central elements are unique, it follows that there is a bijection between $Z(A)$ and $\mathcal{K}_{A}=\{(\vec{e},\vec{f})\in A^{2n}\mid \vec{e}\diamond_{A}\vec{f}\}$. Then, if $f:A\rightarrow B$ is a homomorphism which preserves complementary central elements, it is clear that it also must preserve central elements. Classical examples of varieties with BFC in which every homomorphism preserves complementary central elements are the varieties ${\mathcal{R}}$ of commutative rings with unit and ${\mathcal{L}_{0,1}}$ of bounded distributive lattices. However, we stress that definitions above are not trivial, since there are varieties with BFC with homomorphisms that preserve central elements but does not preserve complementary central elements and even varieties with BFC with homomorphisms that does not preserve nor central elements nor complementary central elements. In order to illustrate the first situation, let $\mslattice$ be the variety of bounded meet semilattices. Since $\mslattice$ is a variety with $0$ and $1$, and the formula $\varphi(x,y,z)= (x\wedge z\approx y\wedge z)$ satisfies the condition (R) of the Introduction, it follows that $\mslattice$ is a variety with BFC. Let us consider the algebras $\textbf{A}=\textbf{2}\times \textbf{2}$ and $\textbf{B}=\textbf{2}\times \textbf{2}\times \textbf{2}$ (with $\textbf{2}$ the chain of two elements). Notice that $Z(A)=A$, $Z(B)=B$ and moreover, $(1,0,0)\diamond_{B}(0,1,1)$, $(0,1,0)\diamond_{B}(1,0,1)$ and $(0,0,1)\diamond_{B}(1,1,0)$. Let $\alpha:\textbf{A}\rightarrow \textbf{B}$ be the homomorphism defined by $\alpha(1,1)=(1,1,1)$, $\alpha(0,0)=(0,0,0)$, $\alpha(0,1)=(1,0,0)$ and $\alpha(1,0)=(0,0,1)$. It is clear that $\alpha$ preserves central elements but does not preserves complementary central elements. For the last situation, let $\mathcal{L}$ be the variety of bounded lattices. It is known (see \cite{V1999} and \cite{FH1970}) that $\mathcal{L}$ is a variety with BFC. If $\textbf{C}=\textbf{2}\times \textbf{2}$ and $\textbf{D}=\{0,1,a,b,c\}$, with $\{a,b,c\}$ not comparables, it easily follows that $\textbf{C}$ is subalgebra of $\textbf{D}$, but $\textbf{C}$ is directly decomposable while $\textbf{D}$ is not. So the inclusion does not preserve nor central elements nor complementary central elements.
\end{rem}

\begin{defi}\label{Stability and C-Stability}
Let $\V$ be a variety with BFC. We say that $\V$ is stable by complements if every homomorphism of $\V$ preserves complementary central elements. 
\end{defi}

Let $\textbf{V}$ be a variety with BFC. We say that a formula $\Sigma(\vec{z},\vec{u})$ \emph{defines the property} $\vec{e}\diamond_{A} \vec{f}$ \emph{in} $\mathbf{V}$ if for every $A\in \mathbf{V}$ and $\vec{e},\vec{f}\in A^{n}$ it follows that $\vec{e}\diamond_{A} \vec{f}$ if and only if $A\models \sigma[\vec{e},\vec{f}]$.

\begin{lem}\label{Characterization in terms of definability}
Let $\textbf{V}$ be a variety with BFC. The following are equivalent:
\begin{itemize}
\item[1.] The relation $\vec{e}\diamond_{A} \vec{f}$ is definable by a formula $\exists \bigwedge p=q$.
\item[2.] $\textbf{V}$ is stable by complements.
\end{itemize}
\end{lem}
\begin{proof}
Let $\mathcal{L}$ be the language of $\textbf{V}$ and consider $\mathcal{L}_{1}=\mathcal{L}\cup \{R\}$, where $R$ is $2n$-ary relation symbol. Observe that a $\mathcal{L}_{1}$-structure is just a pair $(A,R^{A})$, where $A$ is a $\mathcal{L}$-algebra and $R^{A}\subseteq A^{2n}$. For an algebra $A\in \mathcal{V}$, we define $R^{A}=\{(\vec{e},\vec{f})\in A^{2n}\mid \vec{e}\diamond_{A}\vec{f}\}$. Let us regard the following class $\mathcal{K}=\{(A,R^{A})\mid A\in \V\}$ of $\mathcal{L}_{1}$-structures. From Lemma 4.1 of \cite{SV2009}, there exists a set $\Gamma$ of first order formulas in the language of $\V$ such that, for every $A\in \textbf{V}$, $\vec{e}\diamond_{A}\vec{f}$ if and only if $A\models \sigma[\vec{e},\vec{f}]$, for every  $\sigma(\vec{x},\vec{y})\in \Gamma$. Moreover, the aforementioned Lemma, also ensures that $\mathcal{K}$ is stable by finite products. It is easy to see that the class $\mathcal{K}$ is a axiomatizable by the set of sentences \[\Sigma \cup \{(\forall_{\vec{x},\vec{y}})(R(\vec{x},\vec{y})\leftrightarrow \bigwedge_{\sigma \in \Gamma}\sigma(\vec{x},\vec{y}))\}, \] 

where $\Sigma$ is any set of axioms for the class $\textbf{V}$. Recall that, since $\mathcal{L}$ is finite, then $\mathcal{K}$ is a first order class, so it is closed by ultraproducts. Hence, item (4) of Theorem 6.2 of \cite{CV2016} tell us that the following are equivalent: 
\begin{itemize}
\item[(a)] There is a formula  $\exists \bigwedge At(\mathcal{L})$ which defines $R$ in $\mathcal{K}$.
\item[(b)] If $(A,R^{A})$,$(B,R^{B})\in \mathcal{K}$ and $\sigma:A\rightarrow B$ is a homomorphism, then $\sigma:(A,R^{A})\rightarrow (B,R^{B})$ is a homomorphism. (Notice that $ (A,R^{A})_{\mathcal{L}}=A$ for every $\mathcal{L}_{1}$-structure $(A,R^{A})$).
\end{itemize}
But $1.$ is a restatement of $(a)$ and $2.$ is a restatement of $(b)$. This concludes the proof.
\end{proof}

\begin{defi}\label{Coextensivity} A category with finite limits $\mathcal{C}$ is called extensive if has finite coproducts and the canonical functors $1 \rightarrow \mathcal{C}/0$ and $\mathcal{C}/X \times \mathcal{C}/Y \rightarrow \mathcal{C}/(X + Y )$ are equivalences. 
\end{defi}

If the opposite $\mathcal{C}^{op}$ of a category $\mathcal{C}$ is extensive, we will say that $\mathcal{C}$ is coextensive. Classical examples of coextensive categories are the categories ${\Ring}$ and $\dLat$ of commutative rings with unit and bounded distributive lattices, respectively. In the following, we will use a characterization proved in \cite{CW1993}.

\begin{prop}\label{Coextensivity Proposition} 
A category $\mathcal{C}$ with finite coproducts and pullbacks along its injections is extensive if and only if the following conditions hold:
\begin{enumerate}
\item (Coproducts are disjoint.) For every $X$ and $Y$, the injections $X \xrightarrow[]{in_{0}} X+Y$ and $Y\xrightarrow[]{in_{1}} X+Y$ are monic and the square below is a pullback
\begin{displaymath}
\xymatrix{
0 \ar[r]^-{!} \ar[d]^-{!} & Y \ar[d]^-{in_{1}} 
\\
X \ar[r]_-{in_{0}}  & X+Y
}
\end{displaymath}

\item (Coproducts are universal.) For every $X,X_{i}, Y_{i}$ with $i=0,1$ and $X \xrightarrow[]{f} Y_{0}+Y_{1}$, if the squares below are pullbacks
\begin{displaymath}
\xymatrix{
X_{0} \ar[r]^-{x_{0}} \ar[d]_-{h_{0}} & X \ar[d]^-{f} & X_{1} \ar[l]_-{x_{1}} \ar[d]^-{h_{1}}
\\
Y_{0} \ar[r]_-{y_{0}} & Y_{0}+Y_{1} & Y_{1} \ar[l]^-{y_{1}}
}
\end{displaymath}
\noindent
then, the cospan $X_{0}\rightarrow X \leftarrow X_{1}$ is a coproduct.
\end{enumerate}
\end{prop}

Let $\mathbf{V}$ be a variety with BFC. We write $\mathcal{V}$ to denote the algebraic category associated to $\mathbf{V}$.

\begin{lem}\label{preservation imply prods stable by pushouts}
Let $\mathbf{V}$ be a variety with RexDFC. If $\mathbf{V}$ is stable by complements then, in $\mathcal{V}$ the products are stable by pushouts.
\end{lem}
\begin{proof}

Let $A,B\in \textbf{B}$ and $f:A\rightarrow B$ be a homomorphism. If $A\cong A_{1}\times A_{2}$, let us consider de diagram: 

\begin{displaymath}
\xymatrix{
A_{1} \ar[d] & \ar[l] \ar[r] \ar[d] A & A_{2} \ar[d]
\\
P_{1}  & \ar[l] \ar[r]  B & P_{2} 
}
\end{displaymath} 
\noindent
Where $P_{1}$ and $P_{2}$ are the pushouts from the left and the right squares, respectively. If $i$ denotes the isomorphism between $A$ and $A_{1}\times A_{2}$, then $A_{j}\cong A/Ker(\pi_{j}i)$ (with $j=1,2$). If $\vec{e}_{j}$ denotes the central element corresponding to $Ker(\pi_{j}i)$, since $Ker(\pi_{1}i) \diamond Ker(\pi_{2}i)$ in $Con(A)$, then from Lemma \ref{Definability by principal congruences} we have that $Ker(\pi_{1}i)=\theta^{A}(\vec{1},\vec{e}_{2})$ and $Ker(\pi_{2}i)=\theta^{A}(\vec{1},\vec{e}_{1})$. From, item 2${.}$ of Corollary \ref{corollary universal property}, we get that $P_{1}\cong B/\theta^{B}(\vec{1},f(\vec{e}_{2}))$ and $P_{2}\cong B/\theta^{B}(\vec{1},f(\vec{e}_{1}))$. The universal property of pushouts implies that $B\rightarrow P_{1}$ coincides with $B\rightarrow \theta^{B}(\vec{1},f(\vec{e}_{2}))$ and $B\rightarrow P_{2}$ with $B\rightarrow \theta^{B}(\vec{1},f(\vec{e}_{1}))$. Since $f$ preseserves pairs of complementary central elements by assumption, we can conclude that $B\cong B/\theta^{B}(\vec{1},f(\vec{e}_{2})) \times B/\theta^{B}(\vec{1},f(\vec{e}_{1}))\cong P_{1}\times P_{2}$.
\end{proof}

\begin{lem}\label{reciprocal preservation imply prods stable by pushouts}
Let $\mathbf{V}$ be a variety with RexDFC, $A,B\in \mathbf{V}$ and $f:A\rightarrow B$ be a homomorphism. If in $\mathcal{V}$ binary products are stable by pushouts along $f$, thus $f$ preserves pairs of complementary central elements.
\end{lem}
\begin{proof}
Let $A\in \mathcal{V}$ and $\vec{e}$ be a central element of $A$. If $\vec{g}$ denotes the complementary central element of $\vec{e}$, by Lemma \ref{Definability by principal congruences} we get that $\theta^{A}_{\vec{0},\vec{e}}=\theta^{A}(\vec{1},\vec{g})$, and consequently that $A\cong A/\theta^{A}(\vec{1},\vec{g})\times A/\theta^{A}(\vec{1},\vec{e})$. Let us, consider the diagram
\begin{displaymath}
\xymatrix{
A/\theta^{A}(\vec{1},\vec{g}) \ar[d] & \ar[l] \ar[r] \ar[d]^-{f} A & A/\theta^{A}(\vec{1},\vec{e}) \ar[d]
\\
B/\theta^{B}(\vec{1},f(\vec{g})) & \ar[l] \ar[r] B & B/\theta^{B}(\vec{1},f(\vec{e}))
}
\end{displaymath}

By Corollary \ref{corollary universal property} both squares are pushouts, so, since binary products are stable by pushouts along $f$ by assumption, the span $B/\theta^{B}(\vec{1},f(\vec{g}))\leftarrow B\rightarrow B/\theta^{B}(\vec{1},f(\vec{e}))$ is a product. This fact implies directly that $\theta^{B}(\vec{1},f(\vec{g}))\diamond \theta^{B}(\vec{1},f(\vec{e}))$ in $Con(B)$. From Corollary \ref{Centrals are determined by principal congruences}, we obtain that both $f(\vec{e})$ and $f(\vec{g})$ are central elements of $B$. Hence, we conclude that $f(\vec{e})\diamond_{B} f(\vec{g})$. 
\end{proof}

\begin{lem}\label{0 and 1 imply products are codisjoint}
If $\mathbf{V}$ is a variety with $\vec{0}$ and $\vec{1}$, then, in $\mathcal{V}$ the pushout of the projections of binary products is the terminal object. 
\end{lem}
\begin{proof}
Since $\mathcal{V}$ is an algebraic category, for every pair of $A,B\in \mathcal{V}$ the pushout of the projections $A\leftarrow A\times B \rightarrow B$ belongs to $\mathcal{V}$. It is clear that the projections send $[\vec{0},\vec{1}]\in A\times B$ into $\vec{0}$ in $A$ and into $\vec{1}$ in $B$, so $\vec{0}=\vec{1}$ in the pushout. Since $\mathbf{V}$ is a variety with $\vec{0}$ and $\vec{1}$, it follows that the pushout must be the terminal object.
\end{proof}

\begin{theo}\label{scc is equivalent to coextensivity}
Let $\mathbf{V}$ be a variety with RexDFC. The following are equivalent:
\begin{enumerate}
\item The relation $\vec{e}\diamond_{A} \vec{f}$ is definable by a formula $\exists \bigwedge p=q$.
\item $\mathbf{V}$ is stable by complements.  
\item $\mathcal{V}$ is coextensive.
\end{enumerate}
\end{theo}
\begin{proof}
Observe that, from Lemma \ref{Characterization in terms of definability}, it follows that $1.$ is equivalent to $2.$ On the one hand, notice that, since $\mathbf{V}$ has BFC, the variety is a variety with $\vec{0}$ and $\vec{1}$, which by Lemma \ref{0 and 1 imply products are codisjoint}, implies that in $\mathcal{V}$ the pushout of the projections of binary products is the terminal object. Since $\mathcal{V}$ is an algebraic category, it is also clear that the projections are epi. In order to prove $2.$ implies $3.$ let us assume that $\mathbf{V}$ is stable by complements. From Lemma \ref{preservation imply prods stable by pushouts}, binary products are stable by pushouts. Hence, by the dual of Proposition \ref{Coextensivity Proposition}, $\mathcal{V}$ is coextensive. Finally, from Lemma \ref{reciprocal preservation imply prods stable by pushouts} we get that $3.$ implies $2$. 
\end{proof}

We conclude this section by recalling that in \cite{CPR2001} there is a characterization of coextensive algebraic categories. It is worth mentioning that this result is stated in strictly categorical terms. However, for the context in which we are working on this paper, we emphasize that Theorem \ref{scc is equivalent to coextensivity} provides the possibility to a non-categoricist mathematician to solve a categorical problem through algebraic methods, by verifying whether the morphisms in the variety preserve complementary central elements; or even through syntactical methods, by examining if a relation between central elements can be defined through a formula of determined complexity.

\section{RexDFC and stability by complements induce homomorphisms of Boolean algebras}\label{RexDFC and stability by complements induce homomorphisms of Boolean algebras}

As we saw in Section \ref{Coextensivity,  definability and stabitlity by complements}, not every variety with BFC has center stable by complements. In this section we prove that a variety with RexDFC having center stable by complements is in fact a variety with the Fraser-Horn Property. This result will allow us to prove that every homomorphism $f$ in the variety induces a Boolean algebra homomorphism between the centers of $dom(f)$ and $cod(f)$. 

\begin{lem}[Theorem 1 \cite{FH1970}]\label{Theorem 1 FHP}
Let $\textbf{K}$ be a variety and $A$, $B$ be algebras of $\textbf{K}$. The following are equivalent:
\begin{enumerate}
\item $\textbf{K}$ has FHP.
\item For every $A,B\in \textbf{K}$ and $\gamma\in Con(A\times B)$,
\[\Pi_{1}\cap (\Pi_{2} \vee \gamma) \subseteq \gamma\; \textrm{and}\; \Pi_{2}\cap (\Pi_{1} \vee \gamma) \subseteq \gamma\]
where $\Pi_{1}$ is the kernel of the projection on $A$ and $\Pi _{2}$ is the kernel of the projection on $B$.
\end{enumerate} 
\end{lem}

\begin{lem}[Theorem 3 \cite{FH1970}]\label{Theorem 3 FHP}
Let $A$ and $B$ be algebras of the same type. The following are equivalent:
\begin{enumerate}
\item $A\times B$ has FHP.
\item For every $a,c\in A$ and $b,d\in B$,
\[\theta^{A\times B}((a,b),(c,d))=\theta^{A}(a,c)\times \theta^{B}(b,d) \] 
\end{enumerate}
\end{lem}

\begin{lem}\label{useful lema Malsev}
Let $A$ and $B$ be algebras with finite $n$-ary function symbols and $f:A\rightarrow B$ an homomorphism. If $(a,b)\in \theta^{A}(\vec{c},\vec{d})$, then $(f(a),f(b))\in \theta^{B}(f(\vec{c}),f(\vec{d}))$. Thus, if $[\vec{a},\vec{b}]\in \theta^{A}(\vec{c},\vec{d})$ then $[f(\vec{a}),f(\vec{b})]\in \theta^{A}(f(\vec{c}),f(\vec{d}))$.
\end{lem}
\begin{proof}
Apply Lemma \ref{Gratzer Malsev Lemma}. 
\end{proof}

\begin{lem}\label{RexDFC implies FHP}
Let $\V$ be a variety with RexDFC, $A\in \V$ and $\vec{e},\vec{f}\in Z(A)$ such that $\vec{e}\diamond_{A}\vec{f}$. If for every $\gamma\in Con(A)$, $\vec{e}/\gamma \diamond_{A/\gamma} \vec{f}/\gamma$, then $\V$ has the FHP. 
\end{lem}
\begin{proof}
Since $\V$ has RexDFC, from Lemma \ref{Definability by principal congruences},  for every $\vec{e}\in Z(A)$, $\theta^{A}_{\vec{1},\vec{e}}=\theta^{A}(\vec{1},\vec{e})$. So, if $\vec{e}\diamond_{A}\vec{f}$ then $\theta^{A}_{\vec{0},\vec{e}}=\theta^{A}(\vec{1},\vec{f})$. We use Lemma \ref{Theorem 1 FHP}. To do so, we prove $\theta^{A}(\vec{1},\vec{e})\cap (\theta^{A}(\vec{1},\vec{f})\vee \gamma)\subseteq \gamma$. Suppose $(x,y)\in \theta^{A}(\vec{1},\vec{e})\cap (\theta^{A}(\vec{1},\vec{f})\vee \gamma)\subseteq \gamma$, then, $(x,y)\in \theta^{A}(\vec{1},\vec{e})$ and there are $c_{0},...,c_{N}\in A$, with $c_{0}=x$ and $c_{N}=y$, such that $(c_{2i},c_{2i+1})\in \theta^{A}(\vec{1},\vec{f})$ and $(c_{2i+1},c_{2(i+1)})\in \gamma$. Since $A\rightarrow A/\gamma$ is clearly a homomorphism, from Lemma \ref{useful lema Malsev}, we obtain that $(x/\gamma,y/\gamma)\in \theta^{A/\gamma}(\vec{1}/\gamma,\vec{e}/\gamma)$, $(c_{2i}/\gamma,c_{2i+1}/\gamma)\in \theta^{A/\gamma}(\vec{1}/\gamma,\vec{f}/\gamma)$ and $c_{2i+1}/\gamma=c_{2(i+1)}/\gamma$. From transitivity of $\theta^{A/\gamma}(\vec{1}/\gamma,\vec{f}/\gamma)$, we get that $(x/\gamma,y/\gamma)\in \theta^{A/\gamma}(\vec{1}/\gamma,\vec{f}/\gamma)$. Therefore, $(x/\gamma,y/\gamma)\in \theta^{A/\gamma}(\vec{1}/\gamma,\vec{e}/\gamma) \cap \theta^{A/\gamma}(\vec{1}/\gamma,\vec{f}/\gamma)=\Delta^{A/\gamma}$, since $\vec{e}/\gamma \diamond_{A/\gamma} \vec{f}/\gamma$ by assumption, so $(x,y)\in \gamma$. The proof of $\theta^{A}(\vec{1},\vec{f})\cap (\theta^{A}(\vec{1},\vec{e})\vee \gamma)\subseteq \gamma$ is similar. This concludes the proof.
\end{proof}

\begin{coro}\label{RexDFC with SCC has FHP}
Let $\V$ be a variety with RexDFC. If $\V$ is stable by complements then has FHP. 
\end{coro}
\begin{proof}
Inmmediate from Lemma \ref{RexDFC implies FHP}.
\end{proof}

\begin{lem}\label{FHP implies TexDFC}
Every variety $\V$ with FHP is TexDFC.
\end{lem}
\begin{proof}
See Theorem 1 of \cite{V1999}.
\end{proof}

As a straight consequence of Corollary \ref{RexDFC with SCC has FHP} and Lemma \ref{FHP implies TexDFC} we obtain

\begin{coro}\label{RexDFC and SCC implies TexDFC}
Every variety $\V$ with RexDFC and stability by complements is TexDFC.
\end{coro}

\begin{rem}\label{Definability LexDFC}
We stress that Lemma \ref{Definability by principal congruences} has an analogue version for varieties with LexDFC. Namely, \emph{a variety $\V$ with BFC has LexDFC if and only if for every $A\in \mathbf{V}$ and $\vec{e}\in Z(A)$, $\theta^{A}_{\vec{0},\vec{e}}=\theta^{A}(\vec{0},\vec{e})$}. Since the proof is basically the same that one the given for the aforementioned Corollary, we leave the details to the reader. 
\end{rem}

\begin{lem}\label{RexDFC and CS entails Lattice morphisms}
Let $\V$ be a variety with RexDFC stable by complements. If $A,B\in \V$ and $f:A\rightarrow B$ is a homomorphism, then $f|_{Z(A)}:Z(A)\rightarrow Z(B)$ is a Boolean algebra homomorphism.
\end{lem}
\begin{proof}
First of all, observe that from Lemma \ref{Definability by principal congruences}, Remark \ref{Definability LexDFC} and Corollary \ref{RexDFC and SCC implies TexDFC}, we get that for every $\vec{e}\in A$, $\theta^{A}_{\vec{0},\vec{e}}=\theta^{A}(\vec{0},\vec{e})$ and $\theta^{A}_{\vec{1},\vec{e}}=\theta^{A}(\vec{1},\vec{e})$. Now, since $f$ is a homomorphism, it is clear that preserves $\vec{0}$ and $\vec{1}$. So, if $\vec{e}_{1},\vec{e}_{2}\in Z(A)$, and $\vec{a}=\vec{e}_{1}\wedge_{A}\vec{e}_{2}$, thus from Lemma \ref{Useful lema Centrals}, $[\vec{0},\vec{a}]\in \theta^{A}(\vec{0},\vec{e}_{1})$ and $[\vec{a},\vec{e}_{2}]\in \theta^{A}(\vec{1},\vec{e}_{1})$. By Remark \ref{SC and CSC are not trivial}, $f(\vec{e})\in Z(A)$ for every $\vec{e}\in A$, thus from Lemma \ref{useful lema Malsev} we get that $[\vec{0},f(\vec{a})]\in \theta^{A}(\vec{0},f(\vec{e}_{1}))$ and $[f(\vec{a}),f(\vec{e}_{2})]\in \theta^{B}(\vec{1},f(\vec{e}_{1}))$. Hence, again by Lemma \ref{Useful lema Centrals} we can conclude that $f(a)=f(\vec{e}_{1})\wedge_{B}f(\vec{e}_{2})$. The proof for the preservation of the join is similar. This concludes the proof.
\end{proof}

\subsection*{Acknowledgments}

I would like to thank especially to Prof. Diego Vaggione for his
guidance during the study of the Theory of Central Elements and also for his
useful suggestions on this manuscript.


\begin{thebibliography}{20}

\bibitem{BV2013} M. Badano \& D. Vaggione, Varieties with equationally definable factor congruences, Algebra Universalis. Volume 69, 139 --166, 2013.
\bibitem{BB1990} D. Bigelow \& S. Burris. Boolean algebras of factor congruences, Acta Sci. Math., 54, 11--20, 1990.
\bibitem{CV2016} M. Campercholi \& D. Vaggione. Semantical conditions for the definability of functions and relations, Algebra Universalis, Volume 76, Issue 1 71--98, 2016.
\bibitem{CW1993}  A. Carboni, S. Lack, R.F.C. Walters. Introduction to extensive and distributive categories. Journal of Pure and Applied Algebra. Volume 84, Issue 2. 145 --158, 1993.
\bibitem{CPR2001} A. Carboni, M. C. Pedicchio \& J. Rosick\'y. Syntactic characterizations of various classes of locally presentable categories. Journal of Pure and Applied Algebra, Volume 161, Issues 1 -- 2, 65 -- 90, 2001. 

\bibitem{FH1970} G. A. Fraser \& A. Horn, Congruence relations in direct products. Proc. Amer. Math. 26, 390-394, 1970.
\bibitem{MMT1987} R. McKenzie, G. McNulty \& W. Taylor, Algebras, Lattices, Varieties, Vol. 1, Wadsworth \& BrooksrCole Math. Series, Monterey, CA, 1987.
\bibitem{P1967} R. S. Pierce, Modules over Commutative Regular Rings. Memoirs of the American Mathematical Society. 70, 1967.
\bibitem{SV2009} P. Sanchez Terraf \& D. Vaggione, Varieties with definable factor congruences. Trans. Amer. Math. Soc. 361, 5061 -- 5088, 2009.
\bibitem{S2010} P. Sanchez Terraf, Existentially definable factor congruences, Acta Sci. Math. 76, 49-54, 2010.
\bibitem{V1999} D. Vaggione, Central Elements in Varieties with the Fraser -- Horn Property. Advances in Mathematics. Vol. 148, Issue 2, 193 -- 202, 1999. 


\end{thebibliography}
\end{document}